\documentclass[11pt,reqno]{amsart}
\usepackage[T1]{fontenc}
\usepackage{lmodern}
\usepackage{amsmath,amssymb,mathtools}
\usepackage[a4paper,margin=29mm]{geometry}
\usepackage{microtype}
\usepackage{enumitem}
\usepackage{xcolor}
\definecolor{linkblue}{RGB}{30,65,100}
\usepackage[colorlinks=true,linkcolor=linkblue,citecolor=linkblue,urlcolor=linkblue]{hyperref}
\usepackage{bookmark}
\newcommand{\C}{\mathbb C}
\newcommand{\R}{\mathbb R}
\newcommand{\CP}{\mathbb {CP}}
\newcommand{\ddc}{dd^c}
\newcommand{\ddV}{dV_E}
\newcommand{\omegaE}{\omega_E}
\newcommand{\omegaFS}{\omega_{\mathrm{FS}}}
\DeclareMathOperator{\Ric}{Ric}
\DeclareMathOperator{\Vol}{Vol}
\DeclareMathOperator{\diam}{diam}
\DeclareMathOperator{\tr}{tr}
\DeclareMathOperator{\Scal}{Scal}
\DeclareMathOperator{\Tan}{Tan}
\DeclareMathOperator{\AVR}{AVR}
\DeclareMathOperator{\Rm}{Rm}

\newtheorem{theorem}{Theorem}[section]
\newtheorem{proposition}[theorem]{Proposition}
\newtheorem{lemma}[theorem]{Lemma}
\newtheorem{corollary}[theorem]{Corollary}
\theoremstyle{definition}

\theoremstyle{remark}
\newtheorem{remark}[theorem]{Remark}
\numberwithin{equation}{section}
\setlist[enumerate,1]{label=\textup{(\roman*)},leftmargin=*,itemsep=3pt}
\allowdisplaybreaks[2]
\title[Nonunique tangent cones at infinity]
{A complete K\"ahler metric on $\C^n$ of nonnegative Ricci curvature with Euclidean volume growth and nonunique asymptotic cone}
 \author{Jiwei Dai}
\address{School of Mathematical Sciences, East China Normal University, Shanghai 200241, China}
\email{jiweidaimath@gmail.com}
\author{Wancheng Zhang}
\address{School of Mathematical Sciences, Zhejiang University, Hangzhou 310058, China}
\email{12335021@zju.edu.cn}
\begin{document}
\begin{abstract}
We construct complete smooth K\"ahler metrics on $\C^n$, $n\ge2$,
with nonnegative Ricci curvature, Euclidean volume growth, and
nonunique tangent cones at infinity. More generally, we prescribe
the entire tangent  cone set within a class of smooth Hopf cones, using
a compact connected family of positive transverse Ricci forms on
$\CP^{n-1}$.  
We identify all metric blow-downs and their Ricci currents, compute
the exact asymptotic volume ratio, and classify 
isometries between the model cones. In particular, the K\"ahler
condition does not ensure uniqueness at infinity under nonnegative
Ricci curvature and Euclidean volume growth, even in complex dimension two.
\end{abstract}
\maketitle
\tableofcontents
\clearpage

\section{Introduction}
\label{main:section}

Tangent cones at infinity  describe the geometry of a complete noncompact manifold at large scales. For
noncollapsed Gromov--Hausdorff limits of manifolds with a uniform
lower Ricci curvature bound, the work of Cheeger and Colding
\cite{CheegerColding1996,CheegerColding1997} shows that tangent cones
at a point are metric cones. Their theory also implies that
blow-down limits of a complete manifold with nonnegative Ricci
curvature and Euclidean volume growth are metric cones. These
cone-structure results do not imply uniqueness of the limit.

The first nonuniqueness example under positive Ricci curvature and
Euclidean volume growth is due to Perelman \cite{Perelman1997}.
His four-dimensional construction varies the geometry of the
cross sections slowly along the radial direction and has
nonisometric asymptotic cones. Colding and Naber
\cite{ColdingNaber2013} subsequently developed a family-realization
construction for tangent cones of noncollapsed Ricci limit spaces.
Their local realization theorem has equal-volume, Ricci lower
bound, and Ricci-closability hypotheses. Their examples at infinity
include manifolds whose tangent cones have different Euclidean
splitting behavior, and a five-dimensional manifold with
nonhomeomorphic tangent cones
\cite[Theorems~1.1 and~1.4]{ColdingNaber2013}.

Additional geometric equations can enforce uniqueness. For example,
Colding and Minicozzi \cite{ColdingMinicozzi2014} proved uniqueness
at infinity for a Ricci-flat manifold with Euclidean volume growth
when one tangent cone has a smooth cross section. In the
K\"ahler setting, Liu \cite{Liu2024} obtained a rigidity theorem
from the existence of a Ricci-flat tangent cone under additional
curvature-decay assumptions. These results leave open the
possibility of nonuniqueness when all tangent cones have nonzero
Ricci curvature.

Nonuniqueness at infinity already occurs in K\"ahler geometry without
Euclidean volume growth. Hattori \cite{Hattori2017} constructed complete
four-dimensional hyper-K\"ahler manifolds with nonunique asymptotic
limits and non-Euclidean volume growth. These examples lie outside
the noncollapsed setting considered here. Closely related earlier examples in the noncollapsed local setting
were constructed by Karmakar
\cite[Theorem~1.2 and Appendix~A, Theorem~A.1]{Karmakar2026}.
In this paper, Karmakar produced smooth polarized K\"ahler metrics on $\CP^2$, in one fixed
integral K\"ahler class, with a uniform Ricci lower bound and uniform
noncollapse, whose Gromov--Hausdorff limit has a point admitting a
continuous family of pairwise nonisometric tangent cones. These cones
have the same underlying affine variety $\C^2$, the same Hopf/Reeb
action, and the same volume density. His construction uses slowly
varying Hopf cone potentials, a logarithmic radial Ricci correction,
and caps obtained by prescribing radial Ricci potentials. These
ingredients are closely related to those used in the present paper.

Our results concern a single globally smooth complete metric on the
standard complex manifold $\C^n$, with nonnegative Ricci curvature
everywhere and nonunique tangent cones at infinity. More generally,
we prescribe the entire tangent set by an arbitrary nonempty compact
connected family of smooth strictly positive transverse Ricci forms
in a fixed Hopf class. The global filling must preserve nonnegative
Ricci curvature while controlling its error relative to the positive
radial Ricci margin.

In this paper, we construct K\"ahler examples with nonunique
tangent cones at infinity. Our simplest conclusion is the following.

\begin{theorem}[Nonunique tangent cones at infinity]\label{main:nonunique}
For every $n\ge2$ and $a\in(0,1)$, there is a complete smooth
K\"ahler metric on $\C^n$ with nonnegative Ricci curvature,
asymptotic volume ratio $a^n$, and at least two nonisometric
tangent cones at infinity.
\end{theorem}

In particular, on the standard complex manifold $\C^2$ there is a
complete K\"ahler metric with nonnegative Ricci curvature,
fourth-order volume growth, and at least two nonisometric tangent
cones at infinity. Thus the K\"ahler condition, even together with
nonnegative Ricci curvature and Euclidean volume growth, does not
ensure uniqueness at infinity. The examples also have quadratic
curvature decay, as proved in Corollary~\ref{lim:curvature}.

We prove a more general realization theorem. It prescribes the
whole tangent set in a fixed class of cones associated with the
standard Hopf fibration. We first specify the data used in its
statement. Let $Y=\CP^{n-1}$, and let
$\omega_0=\omegaFS$ with $\Ric(\omega_0)=n\omega_0$.
We use the differential and metric normalizations of
Section~\ref{prelim:kahler}; in particular
$g_\omega=\tfrac12\omega(\,\cdot\,,J\,\cdot\,)$.

Fix
\begin{equation}\label{main:parameters}
 \begin{gathered}
 0<a<1,\qquad \kappa=n(1-a),\\
 \mathcal P_\kappa=
 \{\alpha\in C^\infty(Y,\Lambda^{1,1}_{\R}):
 d\alpha=0,\ \alpha>0,\ [\alpha]=\kappa[\omega_0]\}.
 \end{gathered}
\end{equation}
The topology on $\mathcal P_\kappa$ is the $C^\infty$ topology.
For $\alpha\in\mathcal P_\kappa$, Proposition~\ref{geom:transverse} determines $\phi_\alpha$ and a K\"ahler form
$\beta_\alpha=a\omega_0+\ddc_Y\phi_\alpha$ with
\[
 \Ric(\beta_\alpha)=n\beta_\alpha+\alpha.
\]
Let $(C_\alpha,o_\alpha)$ be the completion at the origin of
$(\C^n\setminus\{0\},g_{\omega_\alpha})$, where
\begin{equation}\label{main:frozen}
 \omega_\alpha=\ddc V_\alpha,\qquad
 V_\alpha(z)=|z|^{2a}e^{\phi_\alpha([z])}.
\end{equation}
Proposition~\ref{geom:cone} shows that this completion is a metric cone. A constant added to
$\phi_\alpha$ gives a pointed isometric cone.

For a complete pointed manifold $(M,g,o)$, we use the notation $\Tan_\infty(M,g,o)$
for the set of isometry classes of tangent cones at infinity.

\begin{theorem}\label{main:theorem}
Let $\mathcal A$ be a nonempty compact connected subset of
$\mathcal P_\kappa$. There exists a smooth strictly plurisubharmonic
function $\mathcal U:\C^n\to\R$, with $\mathcal U(0)=0$, such that
$\omega=\ddc\mathcal U$ and $g=g_\omega$ have the following properties.
\begin{enumerate}
\item The metric $g$ is complete, $\Ric(g)\ge0$ everywhere, and
$g$ is a positive constant multiple of the Euclidean metric near zero.
Its Ricci curvature is strictly positive outside a compact set.
\item The entire tangent set is
\[
 \Tan_\infty(\C^n,g,0)=
 \{[C_\alpha,o_\alpha]:\alpha\in\mathcal A\}.
\]
The pointed real isometry equivalences in this set are precisely
projective holomorphic and projective antiholomorphic equivalence,
as specified in Theorem~\ref{geom:classification}.
\item The normalized asymptotic volume ratio is
\[
 \AVR(g):=\lim_{R\to\infty}
 \frac{\Vol_g B_g(0,R)}{v_{2n}R^{2n}}=a^n.
\]
In particular,
\[
 \Vol_g B_g(p,R)\ge a^n v_{2n}R^{2n}
 \quad\text{for every }p\in\C^n,\ R>0.
\]
\item There are $b,L>0$ and a smooth path $\alpha(\tau)$ with
cluster set $\mathcal A$ such that the following assertion holds.
If $r_j\to\infty$ and
\[
 \alpha\!\left(\log\frac{L+\log r_j^2}{L}\right)
       \longrightarrow\alpha_\infty
       \quad\text{in }C^\infty,
 \qquad
 \ell_j=r_j^a(L+\log r_j^2)^{b/2},
\]
then
\[
 (\C^n,\ell_j^{-2}g,0)
       \xrightarrow{\mathrm{pGH}}(C_{\alpha_\infty},o_{\alpha_\infty}),
 \qquad
 \ell_j^{-2}D_{r_j}^*\omega
       \longrightarrow\omega_{\alpha_\infty}
       \quad\text{in }C^\infty_{\mathrm{loc}}(\C^n\setminus\{0\}),
\]
where $D_r(z)=rz$. Moreover
\begin{equation}\label{main:currents}
 D_{r_j}^*\Ric(\omega)\rightharpoonup
 T_{\alpha_\infty}
 :=\ddc\bigl(\kappa\log|z|^2+\psi_{\alpha_\infty}([z])\bigr)
\end{equation}
as currents on all of $\C^n$. Here
$\alpha=\kappa\omega_0+\ddc_Y\psi_\alpha$ and
$\int_Y\psi_\alpha\omega_0^{n-1}=0$.
The locally integrable potential in \eqref{main:currents}
defines the extension of $\pi^*\alpha_\infty$ across zero.
\end{enumerate}
\end{theorem}

We now describe the construction. The compact transverse equation
associates a cone to each positive form $\alpha$. We choose a smooth
path of such forms whose cluster set is the prescribed family.
On the end of $\C^n$, the corresponding potentials are inserted at
the slow time
\[
 \tau=\log\frac{L+\log r^2}{L}.
\]
Their variation tends to zero on every annulus of fixed radius
ratio. The extra factor $(L+\log r^2)^b$ supplies a positive radial
Ricci term of order $(L+\log r^2)^{-2}$. This term controls the
mixed Ricci terms created by the changing transverse geometry.

The remaining step is to fill the origin while preserving this
Ricci margin. We first extend the radial Ricci potential and then
recover a K\"ahler potential by an explicit radial
Monge--Amp\`ere equation. We delay the transverse variation until
the filling error, including the derivatives entering the Ricci
form, is smaller than the radial margin. The metric estimates then
give completeness, a shrinking-core estimate, and control of
minimizing curves. Together with smooth convergence on annuli,
these estimates identify every pointed metric blow-down.

The slow variation of cone links is motivated by
\cite{Perelman1997,ColdingNaber2013}. The K\"ahler construction
requires a global potential and a filling compatible with its
Ricci form. Both are established directly here. The prescribed
data are smooth strictly positive forms for the standard Hopf
action; we make no assertion for arbitrary singular transverse
currents or for other Reeb actions.

Section~\ref{prelim:section} records the K\"ahler conventions and
the compact analytic input. Section~\ref{geom:section} constructs
the model cones and the parameter path, and distinguishes their
pointed isometry classes. Section~\ref{ext:section} proves the
exterior Ricci estimate. Section~\ref{cap:section} constructs the
smooth filling. Section~\ref{lim:section} identifies the metric
and current limits and proves the volume and curvature estimates. 

 \textbf{AI usage.} The main idea of this paper originated with Wenshuai Jiang and Gang Liu, who suggested exploiting symmetry in the Kähler potential and constructing such potentials by solving suitable PDE. Following their suggestion, we work with potentials invariant under the $S^1$ action, which allows us to separate the radial and transverse geometry.   GPT-6.1 Sol was used as an auxiliary tool for mathematical discussions, manuscript organization, and refinement of the exposition and language. The authors take full responsibility for the mathematical content and the final manuscript.
\section{Preliminaries}
\label{prelim:section}

We record the K\"ahler conventions and the compact existence and
uniqueness results used to construct the transverse metrics. We
also recall the Ricci-flat rigidity theorem that restricts the
possible tangent sets.

\subsection{K\"ahler potentials and the Ricci form}
\label{prelim:kahler}

Fix $n\ge2$, put $m=n-1$, and let $Y=\CP^{n-1}$ with
\[
 \omega_0=\omegaFS,\qquad
 \Ric(\omega_0)=n\omega_0,\qquad
 \pi^*\omega_0=\ddc\log|z|^2.
\]
Here $\pi:\C^n\setminus\{0\}\to Y$ is the standard projection.
Throughout this paper,
\begin{equation}\label{main:conventions}
 d^c=\frac i2(\bar\partial-\partial),\quad
 \ddc=i\partial\bar\partial,\quad
 \omegaE=\ddc|z|^2,\quad
 g_\omega=\tfrac12\omega(\,\cdot\,,J\,\cdot\,).
\end{equation}
Thus $g_{\omegaE}$ is the ordinary Euclidean metric.
For $\omega=iH_{j\bar k}\,dz^j\wedge d\bar z^k$,
\[
 dV_{g_\omega}=\det H\,\ddV,\qquad
 \omega^n=2^n n!\det H\,\ddV,\qquad
 \Ric(\omega)=-\ddc\log\det H.
\]
The sign of the Ricci form agrees with that of the real Ricci tensor.
We write $v_{2n}$ for the Euclidean volume of the unit ball in $\R^{2n}$.

A real smooth function $U$ is strictly plurisubharmonic when
$(U_{j\bar k})$ is positive definite; then $dd^cU$ is a K\"ahler
form. For two K\"ahler forms $\gamma_1,\gamma_2$ on a complex
$m$-manifold,
\begin{equation}\label{prelim:ricci-ratio}
 \Ric(\gamma_1)-\Ric(\gamma_2)
 =-dd^c\log\frac{\gamma_1^m}{\gamma_2^m}.
\end{equation}
In particular, $\Ric(c\gamma)=\Ric(\gamma)$ for $c>0$.
The Ricci class is $2\pi c_1$ in unscaled de Rham cohomology.
On $Y$ it equals $n[\omega_0]$.

The $\partial\bar\partial$-lemma on $Y$ identifies the parameter
space in \eqref{main:parameters} with normalized potentials:
\begin{equation}\label{geom:positive-class}
 \alpha=\kappa\omega_0+dd_Y^c\psi_\alpha>0,
 \qquad \int_Y\psi_\alpha\,\omega_0^m=0.
\end{equation}
Thus $\mathcal P_\kappa$ is the set of K\"ahler representatives
of one fixed positive class. It is a convex set of forms.
The $C^\infty$ topology means convergence in every smooth norm;
all compactness assertions for transverse data use this topology.

\subsection{The compact transverse equation}
\label{prelim:twisted}

The transverse Ricci data do not directly specify a metric.
We must solve $\Ric_Y\beta=n\beta+\alpha$ in the class
$a[\omega_0]$. Its cohomological compatibility is
$na+\kappa=n$. The analytic input is the existence criterion of
Sz\'ekelyhidi \cite[Theorem~1]{Szekelyhidi2011} for the compact
twisted K\"ahler--Einstein equation, together with Berndtsson's
twisted uniqueness theorem \cite[Theorem~7.1]{Berndtsson2015}.
We give the rescalings and hypotheses in the following application.

\begin{proposition}[Normalized transverse metrics]
\label{geom:transverse}
For every $\alpha\in\mathcal P_\kappa$ there is a unique K\"ahler
form $\beta_\alpha\in a[\omega_0]$ such that
\begin{equation}\label{geom:twisted-equation}
 \Ric_Y\beta_\alpha=n\beta_\alpha+\alpha.
\end{equation}
Let $\psi_\alpha$ be determined by
\[
 \alpha=\kappa\omega_0+\ddc_Y\psi_\alpha,
 \qquad \int_Y\psi_\alpha\,\omega_0^m=0.
\]
There is a unique normalized potential $\phi_\alpha$ with
\begin{equation}\label{geom:normalized-ma}
 \beta_\alpha=a\omega_0+\ddc_Y\phi_\alpha,
 \qquad
 \log\frac{\beta_\alpha^m}{(a\omega_0)^m}
       +n\phi_\alpha+\psi_\alpha=0.
\end{equation}
The solution depends smoothly on $\alpha$ in H\"older-space local
charts and continuously in the $C^\infty$ topology.
\end{proposition}

\begin{proof}
For existence set $\eta=\alpha/(1-a)$ and $\Omega=(n/a)\beta$.
Both forms lie in $n[\omega_0]$, the Ricci class of $Y$, and the
desired equation becomes
\[
 \Ric\Omega=a\Omega+(1-a)\eta.
\]
By \cite[Theorem~1]{Szekelyhidi2011}, this equation for a prescribed
positive reference form $\eta$ and $a<1$ is solvable whenever there
is a form $\Omega'$ in the Ricci class with
$\Ric\Omega'>a\Omega'$. Here $\Omega'=n\omega_0$ has
$\Ric\Omega'=\Omega'$. This verifies the hypothesis directly.

For uniqueness, $\widehat\Omega=n\beta$ satisfies
$\Ric\widehat\Omega=\widehat\Omega+\alpha$.
The twisted uniqueness theorem \cite[Theorem~7.1]{Berndtsson2015}
relates two solutions by the flow of a holomorphic vector field
$W$ with $W\mathbin{\lrcorner}\alpha=0$. Its assumptions hold:
$Y$ is Fano, $H^{0,1}(Y)=0$, and the twist is smooth and positive.
The theorem is formulated for real line bundles, so no integrality
of $a$ or $\kappa$ is required. Positivity of $\alpha$ forces $W=0$.
The $\partial\bar\partial$-lemma supplies the potentials. In particular,
\[
 \Ric_Y\beta=n\omega_0-\ddc_Y
 \log\frac{\beta^m}{(a\omega_0)^m}.
\]
Equation
\eqref{geom:twisted-equation} makes the left side of
\eqref{geom:normalized-ma} constant, and adding a constant to
$\phi_\alpha$ uniquely makes it zero.

The linearization in $\phi$ is
\begin{equation}\label{geom:linearization}
 L_\alpha v=\Delta_{\beta_\alpha}v+nv,
 \qquad \Delta_\beta v=\beta^{\mu\bar\nu}v_{\mu\bar\nu}.
\end{equation}
It is invertible. Put $\beta=\beta_\alpha$. Since $\alpha>0$,
there is $\epsilon>0$ such that
$\Ric\beta\ge(n+\epsilon)\beta$.
To fix the gradient normalization, define
\[
 \mathcal G_\beta v
 =\beta^{\mu\bar\nu}v_{\bar\nu}
   \frac{\partial}{\partial w^\mu}.
\]
Thus $(\nabla_{g_\beta}v)^{1,0}=2\mathcal G_\beta v$ under
our real metric convention. All tensor norms below are induced
by the Hermitian coefficients of $\beta$.
For a real eigenfunction satisfying
$-\Delta_\beta v=\lambda v$ with $\lambda>0$, the integrated
K\"ahler Bochner identity gives
\[
 0\le
 \int_Y|\bar\partial\mathcal G_\beta v|_\beta^2\,\beta^m
 \le
 \lambda(\lambda-n-\epsilon)\int_Yv^2\,\beta^m.
\]
Consequently every nonzero eigenvalue of $-\Delta_\beta$
is at least $n+\epsilon$, hence greater than $n$.
Constants have eigenvalue $n$ for $L_\alpha$.
Elliptic Fredholm theory and the implicit function theorem yield
smooth local dependence. The normalized map
$\alpha\mapsto\psi_\alpha$ is continuous in all smooth norms,
and uniqueness makes the local maps agree. This proves the assertion
at every H\"older order and hence in $C^\infty$.
\end{proof}

\begin{remark}[The equation on $\CP^1$]\label{prelim:surface}
For $n=2$, write $\alpha=f\omega_0$ and $\beta=u\omega_0$.
The class conditions are
\[
 f,u>0,\qquad
 \int_Yf\omega_0=2(1-a)\int_Y\omega_0,\qquad
 \int_Yu\omega_0=a\int_Y\omega_0.
\]
Since $\Ric_Y(u\omega_0)=2\omega_0-dd^c\log u$, the
transverse equation is a scalar equation on the sphere. In terms
of the normalized potentials it reads
\[
 a+\Delta_{\omega_0}\phi=a e^{-2\phi-\psi}.
\]
If $K_\beta$ is the Gaussian curvature of
$g_\beta=\tfrac12\beta(\,\cdot\,,J\,\cdot\,)$, our convention
gives $\Ric_Y\beta=(K_\beta/2)\beta$. Hence
$K_\beta=4+2f/u>4$. The standard solution is
$\alpha=2(1-a)\omega_0$, $\beta=a\omega_0$.
\end{remark}

\subsection{Blow-downs and Ricci-flat rigidity}
\label{prelim:rigidity}

For clarity, local tangent cones of a Ricci limit space are obtained
by distance rescalings $r_j^{-1}$ with $r_j\downarrow0$.
Tangent cones at infinity of a complete manifold are obtained
with $r_j\to\infty$. For a complete real $N$-manifold with
nonnegative Ricci curvature, Bishop--Gromov comparison makes
\[
 \frac{\Vol B(p,R)}{v_NR^N}
\]
nonincreasing in $R$. A positive limiting ratio is called Euclidean
volume growth. In this case the blow-down limits are noncollapsed
metric cones by \cite{CheegerColding1996,CheegerColding1997}.

Liu and Sz\'ekelyhidi \cite{LiuSzekelyhidi2022} establish
complex-analytic structure for noncollapsed limits of compact
polarized K\"ahler manifolds with a uniform lower Ricci bound and
bounded diameter: the limits are homeomorphic to normal projective
varieties. Their local coordinate theorem gives holomorphic coordinates
on a smaller concentric ball, provided that a sufficiently
larger ball is relatively compact, is sufficiently close to
a Euclidean ball in Gromov--Hausdorff distance, and satisfies
the required Ricci lower bound after rescaling
\cite[Theorem~1.4]{LiuSzekelyhidi2022}.
Here the complex coordinates are fixed globally on $\C^n$;
we identify the metric cones through explicit estimates on annuli.

Liu's rigidity theorem \cite[Theorem~5]{Liu2024} applies to a
complete K\"ahler manifold of complex dimension at least two with
nonnegative Ricci curvature and Euclidean volume growth. It assumes
either quadratic decay of $|\Rm|$, or that the K\"ahler form is
globally $dd^c$-exact and the scalar curvature has quadratic decay.
If one tangent cone at infinity is Ricci flat in the sense of
\cite[Definition~2]{Liu2024}, then the manifold is Ricci flat and
its tangent cone at infinity is unique. The cone hypothesis includes
a nonzero parallel pluricanonical form on the metric regular part,
with the induced limiting complex structure.

For the models below, $\Ric(\omega_\alpha)=\pi^*\alpha$ and
$\alpha>0$, so no model cone is Ricci flat. This is the relevant
distinction: the decay estimate in Corollary~\ref{lim:curvature}
places our examples within the quadratic-decay setting, but the
Ricci-flat-cone hypothesis is absent.

\section{Hopf cones and controlled parameter families}
\label{geom:section}

We first construct a path through the transverse data with a
prescribed cluster set. We then compute the homogeneous cone
metrics, their cross sections, and their   isometry classes. The
strict positivity of the transverse Ricci form will be used both
in the exterior construction and in the isometry classification.

\subsection{A path with prescribed cluster set}

\begin{lemma}[A controlled path with a prescribed limit set]
\label{geom:path}
Let $\mathcal A\subset\mathcal P_\kappa$ be nonempty, compact,
and connected. There is a smooth path
$\alpha:[0,\infty)\to\mathcal P_\kappa$ such that
\begin{enumerate}
\item $\alpha(\tau)=\kappa\omega_0$ for $0\le\tau\le1$;
\item the set of limits of $\alpha(\tau_j)$ for
      $\tau_j\to\infty$ is exactly $\mathcal A$;
\item the range has compact closure in $\mathcal P_\kappa$, and
      $\alpha(\tau)\ge\delta\omega_0$ for some $\delta>0$;
\item for every pair of nonnegative integers $j,k$, the
      $C^k(Y)$ norms of $\partial_\tau^j\alpha$,
      $\partial_\tau^j\psi_\alpha$, and
      $\partial_\tau^j\phi_\alpha$ are uniformly bounded.
\end{enumerate}
In particular $\psi_\alpha=\phi_\alpha=0$ for $0\le\tau\le1$.
\end{lemma}

\begin{proof}
Fix smooth background norms and use the metric
\[
 d_\infty(\gamma,\gamma')
 =\sum_{k=0}^\infty2^{-k-1}
       \min\{1,\|\gamma-\gamma'\|_{C^k}\}
\]
for the $C^\infty$ topology. For every $\varepsilon>0$, any two
points of $\mathcal A$ can be connected by a finite chain in
$\mathcal A$ whose successive distances are less than
$\varepsilon$. To see this, the set of points reachable from a
fixed point by such chains is nonempty and both relatively open
and relatively closed in $\mathcal A$; connectedness proves the
claim. This argument does not require path connectedness.

At stage $q$, visit all points of a finite $2^{-q}$-net of
$\mathcal A$, beginning at the last point of the preceding stage,
using finite $2^{-q}$-chains. Ensure at least one step at each
stage by repeating a point if necessary. The resulting sequence
$\gamma_j\in\mathcal A$ has limit set exactly $\mathcal A$, and
$d_\infty(\gamma_j,\gamma_{j+1})\to0$. In particular,
$\|\gamma_j-\gamma_{j+1}\|_{C^k}\to0$ for each fixed $k$.

Choose a fixed smooth $\chi:[0,1]\to[0,1]$, equal to $0$ near
$0$ and to $1$ near $1$. Join successive nodes on unit intervals by
\[
 \widetilde\alpha(j+t)
 =(1-\chi(t))\gamma_j+\chi(t)\gamma_{j+1}.
\]
The segments glue smoothly. Convexity preserves the class,
closedness, and a uniform positive lower bound. For $\ell\ge1$,
\[
 \partial_\tau^\ell\widetilde\alpha(j+t)
 =\chi^{(\ell)}(t)(\gamma_{j+1}-\gamma_j),
\]
which bounds every mixed derivative. Also
\[
 \sup_{0\le t\le1}
 \|\widetilde\alpha(j+t)-\gamma_j\|_{C^k}\longrightarrow0.
\]
Hence interpolation adds no limit points at infinity. Any sequence
of values either has bounded time parameters and a convergent
subsequence by continuity, or has times tending to infinity and a
convergent subsequence with limit in $\mathcal A$. This proves
compactness of the range closure. Prepending a constant interval
and one smooth convex segment from $\kappa\omega_0$ to
$\gamma_0$ proves the first three assertions without changing the
limit set.

The compact-image assertion of Proposition~\ref{geom:transverse}
gives uniform smooth bounds for $\phi,\psi$, and uniform positivity
of $\beta$. The normalized Poisson equation for $\psi$ gives all
its time derivative bounds. Differentiating
\eqref{geom:normalized-ma} gives
\[
 (\Delta_\beta+n)\dot\phi=-\dot\psi.
\]
The inverses of \eqref{geom:linearization} are uniformly bounded
on the compact parameter set, by local elliptic estimates and
invertibility. Repeated time differentiation and induction give
all the stated bounds for $\phi$. Finally $\phi=\psi=0$ solves
the normalization when $\alpha=\kappa\omega_0$, so uniqueness
proves the last assertion.
\end{proof}

\subsection{Homogeneous Hopf cones}

\begin{proposition}[Homogeneous cone formulas]
\label{geom:cone}
Let $\phi\in C^\infty(Y,\R)$ and
\[
 U(z)=|z|^{2a}e^{\phi([z])},
 \qquad \beta=a\omega_0+\ddc_Y\phi>0.
\]
Then $\omega^C=\ddc U>0$, and
\begin{align}
 \det(U_{j\bar k})
 &=a^2|z|^{2n(a-1)}e^{n\phi}\frac{\beta^m}{\omega_0^m},
 \label{geom:cone-determinant}\\
 \Ric(\omega^C)&=\pi^*(\Ric_Y\beta-n\beta).
 \label{geom:cone-ricci}
\end{align}
The metric completion is a cone with vertex $o$, radial distance
$\rho=\sqrt U$, and dilations $\delta_t(z)=t^{1/a}z$. If $v_{2n}$
is the Euclidean unit ball volume, then
\begin{equation}\label{geom:cone-volume}
 \Vol_{g_{\omega^C}}B(o,R)=a^n v_{2n}R^{2n}.
\end{equation}
\end{proposition}

\begin{proof}
On $z_n\ne0$, write $z=(\lambda w,\lambda)$ and
$H(w)=a\log(1+|w|^2)+\phi(w)$. Then
$U=|\lambda|^{2a}e^H$ and $\beta=iH_{\mu\bar\nu}
dw^\mu\wedge d\bar w^\nu$. Direct differentiation gives
\[
 U_{\lambda\bar\lambda}=\frac{a^2U}{|\lambda|^2},\quad
 U_{\lambda\bar\nu}=\frac{aU}{\lambda}H_{\bar\nu},\quad
 U_{\mu\bar\nu}=U(H_{\mu\bar\nu}+H_\mu H_{\bar\nu}).
\]
The Schur complement of the first entry is $U(H_{\mu\bar\nu})$.
This proves positivity and gives determinant
$a^2U^n|\lambda|^{-2}\det(H_{\mu\bar\nu})$ in these coordinates.
The change to $z$ has Jacobian of absolute value $|\lambda|^m$.
Using $|z|^2=|\lambda|^2(1+|w|^2)$ and
$\det((\omega_0)_{\mu\bar\nu})=(1+|w|^2)^{-n}$ proves
\eqref{geom:cone-determinant}. Taking $-\ddc\log$ of that formula
proves \eqref{geom:cone-ricci}.

Let $V=r\partial_r$ and $Z=\sum z_j\partial_{z_j}$.
Homogeneity and circle invariance imply $ZU=aU$, hence
$\sum_jU_{j\bar k}z_j=aU_{\bar k}$. With our real metric
normalization this gives
\[
 g(V,\cdot)=\frac a2dU,\qquad
 \nabla U=\frac2aV,\qquad |\nabla U|^2=4U.
\]
Thus $|\nabla\rho|=1$ and $\rho\nabla\rho=V/a$.
Moreover $\delta_t^*g=t^2g$ and $\rho\circ\delta_t=t\rho$.
The orthogonal radial flow identifies the metric with
$d\rho^2+\rho^2h$ on $(0,\infty)\times\{\rho=1\}$.
The link is compact; the completion adds exactly one point at
$\rho=0$, and distance to it is $\rho$.

For volume put $J_\beta=\beta^m/\omega_0^m$. Since
$dV_g=\det(U_{j\bar k})dV_E$, polar integration yields
\begin{align*}
 \Vol\{\rho<R\}
 &=\int_{S^{2n-1}}a^2e^{n\phi}J_\beta
   \int_0^{R^{1/a}e^{-\phi/(2a)}}r^{2na-1}\,dr\,dS\\
 &=\frac{aR^{2n}}{2n}\int_{S^{2n-1}}J_\beta\circ\pi\,dS.
\end{align*}
Normalized round measure pushes forward to normalized
Fubini--Study volume, so the average of $J_\beta\circ\pi$ is
$\int_Y\beta^m/\int_Y\omega_0^m=a^m$.
Since $|S^{2n-1}|=2nv_{2n}$, this proves
\eqref{geom:cone-volume}.
\end{proof}

\begin{remark}[The Hopf link]\label{geom:link-metric}
The link in Proposition~\ref{geom:cone} can be described explicitly.
Let
\[
 \eta_0=\left.d^c\log|z|^2\right|_{S^{2n-1}},\qquad
 \eta_\phi=a\eta_0+\pi^*d_Y^c\phi,\qquad
 g_\beta=\tfrac12\beta(\,\cdot\,,J_Y\,\cdot\,).
\]
The map
\[
 \Phi_\phi:S^{2n-1}\longrightarrow\{U=1\},\qquad
 \Phi_\phi(u)=e^{-\phi([u])/(2a)}u
\]
identifies the link metric with
\begin{equation}\label{geom:explicit-link}
 h_\phi=\eta_\phi\otimes\eta_\phi+\pi^*g_\beta,
 \qquad d\eta_\phi=\pi^*\beta.
\end{equation}
Indeed, on the punctured cone put $\eta=d^c\log U$ and
$\rho=\sqrt U$.  Then
\[
 d^cU=\rho^2\eta,\qquad
 dd^cU=2\rho\,d\rho\wedge\eta+\rho^2\pi^*\beta,
 \qquad \Phi_\phi^*\eta=\eta_\phi.
\]
Together with the metric convention in
\eqref{main:conventions}, this gives
\eqref{geom:explicit-link} and
$g_{\omega^C}=d\rho^2+\rho^2h_\phi$.
If $T$ generates $u\mapsto e^{it}u$ with period $2\pi$, then
$\eta_\phi(T)=a$ and $h_\phi(T,T)=a^2$.  Thus every Hopf
fiber has length $2\pi a$; the metric transverse to the fibers is
$g_\beta$.

For example, $\alpha=\kappa\omega_0$ gives
$\phi_\alpha=0$ and $\beta_\alpha=a\omega_0$, so
\[
 h_\alpha=a^2\eta_0\otimes\eta_0+a\pi^*g_{\omega_0}.
\]
For $n=2$ this is a Berger metric on $S^3$.  Its Hopf quotient
is a round $S^2$ of radius $\sqrt a/2$ in our normalization.
\end{remark}

For $\alpha\in\mathcal P_\kappa$, denote this pointed cone by
$C_\alpha$ when $\phi=\phi_\alpha$. Its normalized Ricci potential
and Ricci current are
\begin{equation}\label{geom:cone-current}
 f_\alpha=\kappa\log|z|^2+\psi_\alpha([z])-(n+1)\log a,
 \qquad T_\alpha=\ddc f_\alpha.
\end{equation}
On the punctured space $T_\alpha=\pi^*\alpha$. This formula also
specifies the current at the vertex: $f_\alpha$ is locally integrable
and plurisubharmonic after setting its value at $0$ to $-\infty$.
The coefficients of $\pi^*\alpha$ are $O(r^{-2})$, hence integrable
for $n\ge2$. Integration by parts on $r>\varepsilon$ has boundary
terms $O(\varepsilon^{2n-2}(1+|\log\varepsilon|))$, tending to
zero. Thus $T_\alpha$ is precisely that locally integrable form.

\subsection{Isometry classes and nonunique cones}

\begin{theorem}[Real pointed isometry classification]
\label{geom:classification}
For $\alpha_1,\alpha_2\in\mathcal P_\kappa$, the cones
$C_{\alpha_1}$ and $C_{\alpha_2}$ are   isometric if and
only if either
\begin{equation}\label{geom:equivalence}
\begin{aligned}
 &\alpha_1=P^*\alpha_2
 \quad(P\text{ projective holomorphic}),\\
 &\text{or}\quad \alpha_1=-Q^*\alpha_2
 \quad(Q\text{ projective antiholomorphic}).
\end{aligned}
\end{equation}
Every such isometry on the punctured cones is $z\mapsto Az$ or
$z\mapsto A\bar z$ for some $A\in\operatorname{GL}(n,\C)$.
Cones with different exponents $a$ are not pointed isometric.
\end{theorem}

\begin{proof}
The volume density \eqref{geom:cone-volume} determines $a$.
On either punctured cone put $\mathcal R=\rho\nabla\rho$.
Since $\alpha>0$, the kernel of the real Ricci tensor is exactly
the radial complex line
\[
 \ker\operatorname{Ric}_g
 =\operatorname{span}_{\R}\{\mathcal R,J\mathcal R\}.
\]
A pointed isometry $F$ restricts to a smooth Riemannian isometry,
preserves distance to the vertex, and therefore preserves
$\mathcal R$. Preservation of Ricci and orthogonality imply
$F_*(J_1\mathcal R_1)=\epsilon J_2\mathcal R_2$, with
$\epsilon=\pm1$ constant by connectedness. On a metric cone,
$\nabla\mathcal R=\operatorname{Id}$. Since $J$ is parallel,
$\nabla_X(J\mathcal R)=JX$. Differentiating the identity under
$F$ gives $F_*J_1=\epsilon J_2F_*$. Thus $F$ is holomorphic or
antiholomorphic.

Preserving $\mathcal R$ implies that $F$ commutes with cone
dilations, so $F(tz)=tF(z)$ for $t>0$. A holomorphic $F$ extends
across $0$ (it is bounded there because $\rho(F(z))=\rho(z)$);
its Taylor expansion gives $F(z)=Az$, and its inverse shows that
$A$ is invertible. In the antiholomorphic case composition with
conjugation gives $F(z)=A\bar z$. Naturality of Ricci proves the
necessity of \eqref{geom:equivalence}. The minus sign in the second
relation is the sign for ordinary pullback of real $(1,1)$-forms
under an antiholomorphic isometry.

Conversely, if $\alpha_1=P(A)^*\alpha_2$, uniqueness in
Proposition~\ref{geom:transverse} gives
$\beta_{\alpha_1}=P(A)^*\beta_{\alpha_2}$. Since
$\ddc\log U_\alpha=\pi^*\beta_\alpha$, the function
$\log(U_{\alpha_2}(Az)/U_{\alpha_1}(z))$ descends to a smooth
pluriharmonic function on $Y$, hence is constant. Rescaling $A$ by
a positive scalar makes that constant zero, producing the desired
isometry. If $\alpha_1=-Q^*\alpha_2$, the form
$-Q^*\beta_{\alpha_2}$ solves the same transverse equation for
$\alpha_1$. The identical ratio argument, with
$\ddc(f\circ Q)=-Q^*\ddc f$, gives an antiholomorphic isometry.
\end{proof}

\begin{corollary}[A family giving nonunique cones]
\label{geom:nonunique}
For every $n\ge2$ and $0<a<1$ there is a smooth family
$\alpha_t\in\mathcal P_\kappa$, $0\le t\le t_0$, with
compact connected image such that $C_{\alpha_t}$ is not pointed
isometric to $C_{\alpha_0}$ whenever $0<t\le t_0$.
\end{corollary}

\begin{proof}
Set $\beta_0=a\omega_0$ and choose a real eigenfunction
$\Delta_{\beta_0}h=-\lambda h$ with $\lambda>n/a$.
For sufficiently small $t$ put
\[
 \beta_t=\beta_0+t\ddc_Yh,
 \qquad \alpha_t=\Ric_Y\beta_t-n\beta_t.
\]
Both forms remain strictly positive, and $[\alpha_t]=\kappa[\omega_0]$.
Writing $S(\beta)=\tr_\beta\Ric\beta$, differentiation gives
\[
 \left.\frac d{dt}\right|_{0}S(\beta_t)
 =-\frac na\Delta_{\beta_0}h-\Delta_{\beta_0}^2h
 =\lambda\left(\frac na-\lambda\right)h.
\]
This is nonconstant. Evaluating at two points with different
$h$-values shows that $S(\beta_t)$ is nonconstant for every
sufficiently small positive $t$. The real scalar curvature of the
cone is
\begin{equation}\label{geom:cone-scalar}
 \Scal_{g^C}(z)
 =\frac4{U(z)}\tr_\beta\alpha
 =\frac4{\rho(z)^2}\bigl(S(\beta)([z])-nm\bigr).
\end{equation}
Indeed the inverse Hessian has horizontal block
$U^{-1}\beta^{-1}$ in the adapted horizontal--radial coframe,
and our metric convention gives real scalar curvature
$4\tr_{(U_{j\bar k})}\Ric$. Thus ambient scalar curvature on
the unit distance sphere is constant at $t=0$ and nonconstant for
all sufficiently small positive $t$. A pointed isometry preserves
this sphere and this scalar curvature, proving the assertion.
\end{proof}

\begin{remark}\label{geom:scope}
The classification above concerns cones with the standard Hopf
action and smooth strictly positive transverse Ricci form. Strict
positivity is used to identify the Ricci kernel with the radial
complex line. Neither the real-isometry classification nor the
realization theorem is asserted here for arbitrary semipositive
or singular transverse currents.
\end{remark}

\section{A slowly varying exterior metric}
\label{ext:section}

Fix the path supplied by Lemma~\ref{geom:path}. All of its bounds,
and those of $\psi(\tau)$ and $\phi(\tau)$, are fixed before
choosing $b$ or $L$. In particular, for constants independent of
$\tau$,
\[
 \alpha(\tau)\ge\delta\omega_0,\qquad
 \mu\omega_0\le\beta(\tau)\le M\omega_0,\qquad
 \sup_{\tau\ge0}\|\partial_\tau^j\phi(\tau)\|_{C^k}
 +\sup_{\tau\ge0}\|\partial_\tau^j\psi(\tau)\|_{C^k}<\infty
\]
for every fixed $j,k$. We arrange $\phi=\psi=0$ for $0\le\tau\le1$.
On $|z|\ge1$, put
\begin{equation}\label{ext:variables}
 s=\log|z|^2,\qquad \sigma=L+s,\qquad
 \tau=\log(\sigma/L),\qquad x=\sigma^{-1},\qquad y=[z],
\end{equation}
and define
\begin{equation}\label{ext:potential}
 V(s)=e^{as}\sigma^b,\qquad
 U_{\rm ext}(z)=V(s)e^{\phi(\tau,y)}.
\end{equation}

\begin{proposition}[Exterior estimates]\label{ext:theorem}
There is $b>0$ such that, for every sufficiently large $L$,
$\omega_{\rm ext}=\ddc U_{\rm ext}$ is K\"ahler on $|z|\ge1$.
Its Ricci form satisfies
\begin{equation}\label{ext:ricci-margin}
 \Ric(\omega_{\rm ext})\ge
 c_b\left(\pi^*\omega_0+
          \sigma^{-2}i\theta\wedge\bar\theta\right),
 \qquad \theta=\partial s,\qquad c_b>0.
\end{equation}
The constants can be chosen independently of all sufficiently
large $L$. Moreover,
\begin{align}
 c r^{2a-2}\sigma^b g_E
 &\le g_{\omega_{\rm ext}}
 \le C r^{2a-2}\sigma^b g_E,\label{ext:metric}\\
 c r^{-2\kappa}\sigma^{nb}
 &\le \det((U_{\rm ext})_{j\bar k})
 \le C r^{-2\kappa}\sigma^{nb}.\label{ext:density}
\end{align}
Its Ricci potential has the exact form
\begin{equation}\label{ext:f}
 f_{\rm ext}
 =-\log\det((U_{\rm ext})_{j\bar k})
 =\kappa s-nb\log\sigma+\Psi(\tau,y)+E(s,y),
 \qquad \Psi=\psi-(n+1)\log a.
\end{equation}
For every fixed $k$,
\begin{equation}\label{ext:error}
 \|E\|_{C^k(Y)}\le C_{b,k}\sigma^{-1},\qquad
 \|E_s\|_{C^k(Y)}\le C_{b,k}\sigma^{-2},\qquad
 \|E_{ss}\|_{C^k(Y)}\le C_{b,k}\sigma^{-3}.
\end{equation}
\end{proposition}

\begin{proof}
All angular derivatives hold $s$ fixed, hence also hold $\sigma,\tau$
fixed. In an affine chart on $Y$, use the coframe $(dw,\theta)$.
For any real $S^1$-invariant function $F(s,y)$, the chain rule gives
\begin{align}
 \ddc F={}&F_s\omega_0+\ddc_YF
 +i\partial_YF_s\wedge\bar\theta
 +i\theta\wedge\bar\partial_YF_s
 +iF_{ss}\theta\wedge\bar\theta.\label{ext:chain}
\end{align}
Pullbacks from $Y$ are suppressed in formulas. The first term
comes from $\ddc s=\pi^*\omega_0$; $\theta$ must not be treated
as the differential of a holomorphic coordinate.

Set $h=\log U_{\rm ext}$ and
\[
 q=h_s=a+\frac{b+\phi_\tau}{\sigma},\qquad
 \eta=\partial_Y\phi,\qquad
 \xi=\partial_Yq=\frac{\partial_Y\phi_\tau}{\sigma}.
\]
Since $\sigma_s=1$ and $\tau_s=\sigma^{-1}$,
\begin{equation}\label{ext:q}
 q_s=\frac{\phi_{\tau\tau}-\phi_\tau-b}{\sigma^2}.
\end{equation}
The identity
$\ddc U_{\rm ext}/U_{\rm ext}
=\ddc h+i\partial h\wedge\bar\partial h$
gives the coefficient matrix
\begin{equation}\label{ext:G}
 G=
 \begin{pmatrix}
 B+\eta\eta^*&q\eta+\xi\\
 (q\eta+\xi)^*&q^2+q_s
 \end{pmatrix},\qquad
 B=\beta+\frac{b+\phi_\tau}{\sigma}\omega_0.
\end{equation}
Here and below a form in a matrix denotes its Hermitian coefficient
matrix. Norms of such matrices use $\operatorname{diag}(\omega_0,1)$.

Put $P=b+\phi_\tau$, $T=\phi_{\tau\tau}-\phi_\tau-b$ and
$\zeta=\partial_Y\phi_\tau$. Treating $x,\tau$ as independent
variables, the following is an exact polynomial identity:
\begin{equation}\label{ext:polynomial}
\begin{split}
 G&=G_0+xG_1+x^2G_2,\\
 G_0&=\begin{pmatrix}\beta+\eta\eta^*&a\eta\\a\eta^*&a^2\end{pmatrix},
 \qquad
 G_1=\begin{pmatrix}
 P\omega_0&P\eta+\zeta\\(P\eta+\zeta)^*&2aP
 \end{pmatrix},\\
 G_2&=\begin{pmatrix}0&0\\0&P^2+T\end{pmatrix}.
\end{split}
\end{equation}
For column vectors $v,t$ of the indicated dimensions,
\[
 (v,t)^*G_0(v,t)=v^*\beta v+|\eta^*v+at|^2.
\]
Uniform positivity of $\beta$, boundedness of $\eta$, and $a>0$
imply $c_0I\le G_0\le C_0I$. For example,
\[
 |t|^2\le\frac2{a^2}|\eta^*v+at|^2+
 \frac{2|\eta|^2}{a^2}|v|^2
\]
proves the lower bound. For fixed $b$, all fixed-order derivatives
of $G_1,G_2$ in $\tau,y$ are bounded. There is therefore $x_b>0$
such that $G\ge c_0I/2$ for $0\le x\le x_b$, uniformly in
$\tau,y$. The corresponding upper bound also holds. Choose
$L\ge\max\{1,x_b^{-1}\}$. Since
\[
 \omegaE=e^s(\pi^*\omega_0+i\theta\wedge\bar\theta),
\]
these bounds and boundedness of $\phi$ prove \eqref{ext:metric}.
They also show that $G,G^{-1}$ and all their fixed-order
$s,y$ derivatives are uniformly bounded.

For the determinant, use $z_n=\upsilon$, $w^j=z_j/\upsilon$.
Then
\[
 s=\log|\upsilon|^2+\log(1+|w|^2),\qquad
 dw^1\wedge\cdots\wedge dw^m\wedge\theta
 =\upsilon^{-n}dz^1\wedge\cdots\wedge dz^n.
\]
Since $\det\omega_0=(1+|w|^2)^{-n}$, the change of coframe gives
\begin{equation}\label{ext:det}
 \det((U_{\rm ext})_{j\bar k})
 =U_{\rm ext}^ne^{-ns}D,\qquad
 D=\frac{\det G}{\det\omega_0},\qquad
 D_0=\frac{\det G_0}{\det\omega_0}
 =a^2\frac{\det\beta}{\det\omega_0}.
\end{equation}
The normalized transverse equation yields
\[
 \log D_0=(n+1)\log a-n\phi-\psi.
\]
Taking the negative logarithm of \eqref{ext:det} therefore gives
\eqref{ext:f}, with
$E=-\log(D/D_0)$; the terms involving $n\phi$ cancel exactly.

For derivative estimates, define
\[
 \widetilde E(x,\tau,y)=
 -\log\frac{\det(G_0+xG_1+x^2G_2)}{\det G_0}.
\]
It and all its fixed-order derivatives are uniformly bounded for
$0\le x\le x_b$, and $\widetilde E(0,\tau,y)=0$.
The integral Taylor formula shows that every fixed $\tau,y$
derivative without an $x$ derivative is $O_b(x)$.
Along \eqref{ext:variables},
\[
 \partial_s=-x^2\partial_x+x\partial_\tau.
\]
Consequently
\begin{align*}
 E_s&=-x^2\widetilde E_x+x\widetilde E_\tau,\\
 E_{ss}&=2x^3\widetilde E_x+x^4\widetilde E_{xx}
 -2x^3\widetilde E_{x\tau}
 -x^2\widetilde E_\tau+x^2\widetilde E_{\tau\tau}.
\end{align*}
Because $\widetilde E_\tau,\widetilde E_{\tau\tau}=O_b(x)$,
these identities prove \eqref{ext:error}, also after angular
differentiation. This argument supplies derivative bounds;
they are not obtained by formal differentiation of an $O$ term.

Apply \eqref{ext:chain} to $f_{\rm ext}$. Its horizontal,
mixed and radial blocks are
\begin{equation}\label{ext:ricci-blocks}
\begin{split}
 H&=\alpha-\frac{nb-\Psi_\tau}{\sigma}\omega_0
       +E_s\omega_0+\ddc_YE,\\
 M&=\frac{\partial_Y\Psi_\tau}{\sigma}+\partial_YE_s,\\
 R&=\frac{nb+\Psi_{\tau\tau}-\Psi_\tau}{\sigma^2}+E_{ss}.
\end{split}
\end{equation}
Set
\[
 A_0=\sup|\Psi_\tau|,\quad
 A_1=\sup|\partial_Y\Psi_\tau|_{\omega_0},\quad
 A_2=\sup|\Psi_{\tau\tau}-\Psi_\tau|.
\]
These constants precede both $b$ and $L$. Choose $b>0$ with
\begin{equation}\label{ext:b}
 nb\ge A_2+\frac{2(A_1+1)^2}{\delta}+2.
\end{equation}
For this fixed $b$, choose $C_b$ dominating the relevant errors
in \eqref{ext:ricci-blocks}. Increase $L$ so that, in addition
to the preceding requirements,
\[
 L\ge C_b,\qquad
 \frac{nb+A_0+2C_b}{L}\le\frac\delta2.
\]
Then
\[
 H\ge\frac\delta2\omega_0,\qquad
 |M|\le\frac{A_1+1}{\sigma},\qquad
 R\ge\frac{nb-A_2-1}{\sigma^2}.
\]
The Schur complement satisfies
\begin{equation}\label{ext:schur}
 R-M^*H^{-1}M
 \ge\frac{nb-A_2-1-2(A_1+1)^2/\delta}{\sigma^2}
 \ge\sigma^{-2}.
\end{equation}
For the quantitative form \eqref{ext:ricci-margin}, conjugate
the Ricci matrix by $\operatorname{diag}(I,\sigma)$.
Its horizontal block is uniformly positive, its mixed block is
uniformly bounded, and its radial Schur complement is at least one.
Completing the square shows that the conjugated matrix is bounded
below by $c_bI$. Indeed the substitution
$v\mapsto v+H^{-1}(\sigma M)t$ and its inverse have uniformly
bounded norms. This proves \eqref{ext:ricci-margin}.
Finally, \eqref{ext:f} and boundedness of $\Psi,E$ imply
\eqref{ext:density}.
\end{proof}

\begin{remark}\label{ext:parameters-remark}
All estimates hold uniformly for every $L$ above the stated thresholds.
Increasing $L$ changes the argument $\tau$ of the fixed path but
does not change any path bound. The radial Ricci margin is of
order $\sigma^{-2}$, so the interior correction must be controlled
through the derivatives that enter $\ddc\log\det$ at a smaller
order. The next section establishes precisely that control.
\end{remark}

\section{A smooth K\"ahler filling}
\label{cap:section}

We now fill the origin of the exterior model by a smooth cap. We first
extend the radial Ricci potential as a nondecreasing convex function,
and then recover the K\"ahler potential by an explicit radial
Monge--Amp\`ere equation.  The resulting error at infinity is exponentially
small in the logarithmic radius.  Throughout this section, the path and
$b>0$ have already been fixed.  Constants denoted by $C_b$ or $C_{b,j,k}$
are independent of $L$ once $L$ exceeds a fixed lower bound; they may also
depend on $n,a$ and the fixed path.  We retain the notation of
\eqref{ext:variables}--\eqref{ext:potential} and arrange that
\begin{equation}\label{cap:stationary}
 \phi(\tau,y)=\psi(\tau,y)=0\qquad(0\leq\tau\leq1).
\end{equation}
Increasing $L$ below does not change this path.

\subsection{The radial filling}

\begin{lemma}[The radial cap]\label{cap:radial}
For all sufficiently large $L$, there is a smooth radial strictly
plurisubharmonic function $U_c$ on $\C^n$ whose K\"ahler metric has
nonnegative Ricci curvature.  It equals $\lambda_L|z|^2$ near the origin,
where $\lambda_L>0$.  Writing $U_c=U_c(s)$ away from the origin and
$V(s)=e^{as}(L+s)^b$, there is a real constant $B_L$ such that
\begin{equation}\label{cap:first-integral}
 (U_c'(s))^n=(V'(s))^n+B_L\quad(s\geq0),
 \qquad |B_L|\leq C_b L^{nb}.
\end{equation}
\end{lemma}

\begin{proof}
For a radial potential $W(s)$, where $s=\log|z|^2$, the Euclidean
complex Hessian has eigenvalues $e^{-s}W'$ with multiplicity $n-1$ and
$e^{-s}W''$ with multiplicity one.  Indeed, with $r^2=|z|^2$,
\begin{equation}\label{cap:radial-hessian}
 W_{j\bar k}
 =\frac{W'}{r^2}\delta_{jk}
   +\frac{W''-W'}{r^4}\bar z_jz_k,
 \qquad
 \det(W_{j\bar k})=e^{-ns}(W')^{n-1}W''.
\end{equation}
For $s\geq-2$, put $\sigma=L+s$ and
\[
 p=a+\frac b\sigma,\qquad
 k=p^2-\frac b{\sigma^2}.
\]
Then $V'=Vp$, $V''=Vk$, and the radial Ricci potential is
\begin{equation}\label{cap:radial-f}
 f_{\rm rad}(s)
 :=-\log\det(V_{j\bar k})
 =\kappa s-nb\log\sigma-(n-1)\log p-\log k.
\end{equation}
Here we used $n(1-a)=\kappa$.  For sufficiently large $L$, $p$ and $k$
have uniform positive lower and upper bounds on $s\geq-2$.
Moreover, their logarithmic derivatives satisfy
\[
 (\log p)',(\log k)'=O_b(\sigma^{-2}),\qquad
 (\log p)'',(\log k)''=O_b(\sigma^{-3}).
\]
Consequently
\begin{equation}\label{cap:radial-f-derivatives}
 f_{\rm rad}'=\kappa-\frac{nb}{\sigma}+O_b(\sigma^{-2}),
 \qquad
 f_{\rm rad}''=\frac{nb}{\sigma^2}+O_b(\sigma^{-3}).
\end{equation}
Since $\kappa,b>0$, increasing $L$ makes both derivatives positive on
$s\geq-2$.

Choose a smooth nondecreasing function $\chi:\R\to[0,1]$ which is zero
on $(-\infty,-2]$ and one on $[-1,\infty)$.  Define
\[
 v(s)=
 \begin{cases}
  0,&s\leq-2,\\
  \chi(s)f_{\rm rad}'(s),&s>-2,
 \end{cases}
 \qquad
 f_c(s)=f_{\rm rad}(0)+\int_0^s v(t)\,dt.
\]
This piecewise definition does not require evaluating $f_{\rm rad}$
where $L+s$ is nonpositive.  Flatness of $\chi$ at $-2$ implies that
$v$ and $f_c$ are smooth on $\R$.  We have
\begin{equation}\label{cap:convex-f}
 f_c'\geq0,\qquad f_c''\geq0,\qquad
 f_c=f_{\rm rad}\ \text{on }[-1,\infty).
\end{equation}
Indeed, in the transition interval,
$f_c''=\chi'f_{\rm rad}'+\chi f_{\rm rad}''\geq0$.
The function $f_c$ is a constant $c_L$ on $(-\infty,-2]$.

Define
\begin{equation}\label{cap:ode}
 Y(s)=\left(n\int_{-\infty}^s e^{nt-f_c(t)}\,dt\right)^{1/n},
 \qquad U_c(s)=\int_{-\infty}^sY(t)\,dt.
\end{equation}
Both integrals converge at their lower endpoints.  They give
\[
 U_c'=Y>0,\qquad
 U_c''=Y'=\frac{e^{ns-f_c(s)}}{Y^{n-1}}>0,
 \qquad
 e^{-ns}(U_c')^{n-1}U_c''=e^{-f_c}.
\]
Thus \eqref{cap:radial-hessian} proves strict plurisubharmonicity away
from the origin.  In the adapted coframe with $\theta=\partial s$,
\begin{equation}\label{cap:radial-ricci}
 \Ric(dd^cU_c)=dd^cf_c
 =f_c'\omega_0+i f_c''\theta\wedge\bar\theta\geq0.
\end{equation}
For $s\leq-2$, direct integration in \eqref{cap:ode} gives
\[
 Y(s)=e^{-c_L/n}e^s,
 \qquad U_c(s)=e^{-c_L/n}e^s.
\]
Hence $U_c(z)=\lambda_L|z|^2$ near zero, with
$\lambda_L=e^{-c_L/n}>0$, proving smoothness, positivity, and Ricci
nonnegativity also at the origin.

On $s\geq0$, equality of $f_c$ and $f_{\rm rad}$ gives
\[
 \frac{d}{ds}Y^n=n e^{ns-f_{\rm rad}}
 =\frac{d}{ds}(V')^n.
\]
Their difference is the constant $B_L$ in
\eqref{cap:first-integral}.  To bound it uniformly in $L$, observe that
$f_c(t)-f_{\rm rad}(0)$ is uniformly bounded for all $t\leq0$:
its derivative is supported in $[-2,0]$ there and is uniformly bounded
by \eqref{cap:radial-f-derivatives}.  Also,
\[
 e^{-f_{\rm rad}(0)}=L^{nb}p(0)^{n-1}k(0)\asymp_b L^{nb}.
\]
Therefore
\[
 Y(0)^n=n\int_{-\infty}^0e^{nt-f_c(t)}\,dt\leq C_bL^{nb},
 \qquad (V'(0))^n=L^{nb}(a+b/L)^n\leq C_bL^{nb}.
\]
This proves the asserted bound for $B_L$.
\end{proof}

\subsection{Decay of the filling error}

\begin{lemma}[Size of the cap correction]\label{cap:error}
Assume $n\geq2$, and set
\begin{equation}\label{cap:activation}
 S_0=(e-1)L,\qquad
 \varepsilon_L(s)=|B_L|e^{-nas}(L+s)^{-nb}.
\end{equation}
For sufficiently large $L$, there is a constant $C_L\in\R$ such that
$U_c-V=C_L+h$ on $s\geq S_0$, where, for every fixed integer $j\geq0$,
\begin{equation}\label{cap:h-estimate}
 |h^{(j)}(s)|
 \leq C_{b,j}|B_L|e^{-(n-1)as}(L+s)^{-(n-1)b}.
\end{equation}
In addition,
\begin{equation}\label{cap:weighted-error}
 \sup_{s\geq S_0}(L+s)^2\varepsilon_L(s)
 \leq C_b L^2 e^{-na(e-1)L}\longrightarrow0.
\end{equation}
\end{lemma}

\begin{proof}
Put $Z=B_L/(V')^n$.  The bounds $V'\asymp_b e^{as}\sigma^b$ and
\eqref{cap:first-integral} imply
\[
 |Z(s)|\leq C_b\varepsilon_L(s)
 \leq C_b e^{-nas}(L/\sigma)^{nb}.
\]
Thus $|Z|\leq1/2$ on $s\geq S_0$ when $L$ is large.  Every fixed
derivative of $Z$ is $O_b(\varepsilon_L)$, because every positive-order
derivative of $\log V'=as+b\log\sigma+\log p$ is uniformly bounded.
The function
\[
 A(z)=\begin{cases}
 ((1+z)^{1/n}-1)/z,&z\ne0,\\
 1/n,&z=0
 \end{cases}
\]
is smooth with bounded derivatives on $[-1/2,1/2]$.  Taking the positive
$n$th root in \eqref{cap:first-integral} gives the exact identity
\begin{equation}\label{cap:correction-derivative}
 (U_c-V)'=B_L(V')^{1-n}A(Z).
\end{equation}
Differentiating this identity any fixed number of times proves
\eqref{cap:h-estimate} for $j\geq1$.  The right-hand side of
\eqref{cap:correction-derivative} is integrable at infinity since
$(n-1)a>0$.  Hence $U_c-V$ has a finite limit $C_L$, and
\[
 h(s)=-\int_s^\infty (U_c-V)'(t)\,dt.
\]
Since $(L+t)^{-(n-1)b}$ decreases in $t$, this formula proves the same
bound for $j=0$.

Finally,
\[
 \sigma^2\varepsilon_L(s)
 \leq C_bL^{nb}e^{-nas}(L+s)^{2-nb}.
\]
For large $L$, the logarithmic derivative of the expression on the
right with respect to $s$ is
$-na+(2-nb)/(L+s)<0$.  Its maximum on $[S_0,\infty)$ is therefore
at $S_0$, where $L+S_0=eL$.  This gives
\eqref{cap:weighted-error}.
\end{proof}

\subsection{The global potential}

\begin{proposition}[The global K\"ahler metric]\label{cap:global}
After first fixing the path, then choosing $b$ as in
Proposition~\ref{ext:theorem}, and finally choosing $L$ sufficiently large,
there is a smooth strictly plurisubharmonic function $\mathcal U$ on
$\C^n$ such that $\omega=dd^c\mathcal U$ has nonnegative Ricci curvature
everywhere and strictly positive Ricci curvature outside a compact set.
The potential equals $\lambda_L|z|^2$ near the origin.  On $s>S_0$,
\begin{equation}\label{cap:global-asymptotic}
 \mathcal U=U_{\rm ext}+C_L+h(s),
\end{equation}
with $h$ as in Lemma~\ref{cap:error}.  On the same end,
\begin{equation}\label{cap:end-comparison}
 g_\omega\asymp r^{2a-2}(L+\log r^2)^b g_E,
 \qquad
 \det(\mathcal U_{j\bar k})\asymp
 r^{-2\kappa}(L+\log r^2)^{nb}.
\end{equation}
\end{proposition}

\begin{proof}
For $z\ne0$, define
\begin{equation}\label{cap:global-potential}
 \mathcal U(z)=
 \begin{cases}
 U_c(s),&s\leq S_0,\\[1mm]
 U_c(s)+V(s)\bigl(e^{\phi(\tau,[z])}-1\bigr),&s>S_0,
 \end{cases}
 \qquad \mathcal U(0)=0.
\end{equation}
At $s=S_0$ one has $\tau=1$.  Since the smooth function $\phi$ vanishes
for $0\leq\tau\leq1$, every jet of
$e^{\phi(\tau,y)}-1$ vanishes at $\tau=1$.  Thus
\eqref{cap:global-potential} is a smooth gluing along the sphere
$s=S_0$.  Smoothness and strict plurisubharmonicity at the origin follow
from the exact quadratic expression in Lemma~\ref{cap:radial}.
Equation~\eqref{cap:global-asymptotic} follows from the definition.

Here are the estimates that preserve positivity on the end; in
particular, this step uses derivatives of the metric error, not only its
pointwise size.  In the adapted coframe $(dw,\theta)$, write
\[
 dd^cU_{\rm ext}\longleftrightarrow U_{\rm ext}G.
\]
The exterior computation gives uniform ellipticity of $G$ relative to
$\operatorname{diag}(\omega_0,1)$ and bounds on all its fixed mixed
$s,y$ derivatives.  Since $h$ is radial,
\[
 dd^ch=h'\omega_0+i h''\theta\wedge\bar\theta.
\]
Consequently
\begin{equation}\label{cap:matrix-error}
 dd^c\mathcal U\longleftrightarrow U_{\rm ext}(G+Q),
 \qquad
 Q=\frac{e^{-\phi}}V
 \begin{pmatrix}h'\omega_0&0\\0&h''\end{pmatrix}.
\end{equation}
For every fixed $j,k\geq0$, Lemma~\ref{cap:error} and the uniform
derivative bounds on $\phi$ imply
\begin{equation}\label{cap:Q-derivatives}
 \|\partial_s^j\nabla_Y^k Q\|\leq C_{b,j,k}\varepsilon_L(s).
\end{equation}
All norms here are taken in a fixed finite collection of background
charts, equivalently using $\omega_0$; the estimates are uniform in
$L$.  Increasing $L$ makes $Q$ uniformly small, so $G+Q$ is positive
definite and uniformly comparable to $G$.

Define the change in Ricci potential by
\begin{equation}\label{cap:J}
 J=-\log\frac{\det(\mathcal U_{j\bar k})}
                    {\det((U_{\rm ext})_{j\bar k})}
   =-\log\frac{\det(G+Q)}{\det G}.
\end{equation}
The change-of-coframe determinants and the factor $U_{\rm ext}^n$
cancel in this ratio.  Uniform ellipticity of $G$ and
\eqref{cap:Q-derivatives} give
\begin{equation}\label{cap:J-derivatives}
 |\partial_s^j\nabla_Y^kJ|
 \leq C_{b,j,k}\varepsilon_L(s)
 \qquad(j,k\geq0\text{ fixed}).
\end{equation}
For completeness, the map
$(G,Q)\mapsto-\log\det(G+Q)+\log\det G$ is smooth on the relevant
uniformly elliptic set and vanishes identically when $Q=0$.
Its integral Taylor formula in $Q$, followed by the product and chain
rules, proves \eqref{cap:J-derivatives}: every term contains $Q$ or a
derivative of $Q$.  Only derivatives of $h$ through order four are
needed to estimate $dd^cJ$, all of which were proved above.

The matrix of $dd^cJ$ in the adapted coframe has blocks
\[
 H_J=J_s\omega_0+dd_Y^cJ,\qquad
 M_J=\partial_YJ_s,\qquad R_J=J_{ss}.
\]
Each is $O_b(\varepsilon_L)$.  Conjugating this matrix by
$\operatorname{diag}(I,\sigma)$ multiplies its mixed block by $\sigma$
and its radial block by $\sigma^2$.  Thus its operator norm in this
weighted frame is bounded by
\begin{equation}\label{cap:weighted-J}
 C_b\sigma^2\varepsilon_L(s).
\end{equation}
In exactly this frame the exterior Ricci lower bound
\eqref{ext:ricci-margin} has a uniform positive lower bound $c_b$.
By \eqref{cap:weighted-error}, choose $L$ large enough that
\eqref{cap:weighted-J} is at most $c_b/2$ for every $s\geq S_0$.
Since
\[
 \Ric(dd^c\mathcal U)=\Ric(dd^cU_{\rm ext})+dd^cJ,
\]
we obtain strict Ricci positivity on this end.  On $s\leq S_0$, the
metric is the radial cap and \eqref{cap:radial-ricci} gives Ricci
nonnegativity.  Positivity at the joining sphere also follows by
smoothness and the radial formula there.

Finally, the comparison of $G+Q$ with $G$ proves the metric estimate in
\eqref{cap:end-comparison}.  The determinant ratio in \eqref{cap:J} is
uniformly bounded above and below, so the exterior determinant estimate
gives the second assertion.  The cap depends on $L$, but its sole
contribution to the small-error estimate is bounded by
$|B_L|\leq C_bL^{nb}$.  Formula~\eqref{cap:weighted-error} therefore
justifies choosing $L$ last, with no circular dependence of constants.
\end{proof}

\section{Metric blow-downs and asymptotic geometry}
\label{lim:section}

Let \(\omega=\ddc\mathcal U\) be the global metric constructed in
Proposition~\ref{cap:global}, and write \(g=g_\omega\). All constants in
this section may depend on the fixed choices of \(b,L\) and the path.
For \(r\ge1\), put
\begin{equation}\label{lim:scales}
 \sigma(r)=L+\log r^2,\qquad
 w(r)=r^{a-1}\sigma(r)^{b/2},\qquad
 \rho(r)=r^a\sigma(r)^{b/2}.
\end{equation}
Here $\rho(r)$ is a comparison scale for the constructed metric,
distinct from the exact distance $\sqrt{U_\alpha}$ on a frozen cone.
Its comparison with $d_g(0,z)$ is proved below.
Enlarging constants across a fixed compact annulus, the construction gives
\begin{equation}\label{lim:comparison}
 c w(r)^2g_E\le g\le Cw(r)^2g_E\quad(r\ge1),\qquad
 c r^{-2\kappa}\sigma(r)^{nb}\ddV
 \le dV_g\le C r^{-2\kappa}\sigma(r)^{nb}\ddV.
\end{equation}
The convention \(g_\omega=\tfrac12\omega(\,\cdot\,,J\,\cdot\,)\)
is relevant here: \(dV_g=\det(\mathcal U_{j\bar k})\ddV\), with no
additional constant.

\subsection{Distance and volume control}

\begin{proposition}[Completeness and radial control]\label{lim:complete}
The metric \(g\) is complete. There are constants \(c,C,D>0\) such that,
for \(|z|\ge1\),
\begin{equation}\label{lim:distance}
 c\rho(|z|)-D\le d_g(0,z)\le C\rho(|z|)+D.
\end{equation}
For \(T\ge1\), the intrinsic diameter of the Euclidean ball
\(\overline B_T^E\) and its volume satisfy
\begin{equation}\label{lim:core-bounds}
 \diam_g^{\mathrm{intr}}(\overline B_T^E)
 \le D+C\rho(T),\qquad
 \Vol_g(B_T^E)\le D+C\rho(T)^{2n}.
\end{equation}
\end{proposition}
\begin{proof}
Direct differentiation gives
\begin{equation}\label{lim:rho-derivative}
 \rho'(r)=\left(a+\frac b{\sigma(r)}\right)w(r),\qquad
 a\le a+\frac b{\sigma(r)}\le a+\frac bL.
\end{equation}
Consequently \(\int_1^r w(t)\,dt\) is comparable to
\(\rho(r)-\rho(1)\). A radial segment, followed inside \(B_1^E\)
by a path of uniformly bounded length, gives the upper bound in
\eqref{lim:distance}. For the lower bound, along any piecewise smooth
curve outside \(B_1^E\), with radial coordinate \(r(t)\),
\[
 |\dot\gamma(t)|_g\ge\sqrt c\,w(r(t))|\dot r(t)|.
\]
Apply this to the portion after the last crossing of \(\partial B_1^E\)
and use the fundamental theorem of calculus. This proves the lower bound
for arbitrary connecting curves, without assuming their radiality.
Since \(\rho(r)\to\infty\), closed metric balls lie in compact Euclidean
sets. The metric is smooth and positive on those sets, so they are compact
also in the metric topology. Hopf--Rinow therefore gives completeness.

Joining two points by radial segments and paths through a fixed compact
core gives the intrinsic diameter estimate; all these paths can be chosen
inside \(\overline B_T^E\). Finally, the volume comparison in
\eqref{lim:comparison} and \(n-\kappa=na\) reduce the volume estimate to
\[
 \int_1^T r^{2na-1}\sigma(r)^{nb}\,dr\le C\rho(T)^{2n}.
\]
This follows either by differentiating \(\rho^{2n}\), or directly from
\eqref{lim:rho-derivative}. The fixed core contributes only a constant.
\end{proof}

We record the elementary distance argument used below, so that convergence
of tensors on annuli is not confused with convergence of their completions.
\begin{lemma}[Shrinking-core distance lemma]\label{lim:core-lemma}
Let \(M=\overline B_A^E\). Suppose \(h_j,h\) are smooth positive
metrics on \(M\setminus\{0\}\), their compact intrinsic length completions
are identified with \(M\), and \(h_j\to h\) smoothly on every closed
annulus. The metrics \(h_j\) are allowed to be smooth at zero. Denote the
completed intrinsic distances by \(d_j,d\). If
\[
 \lim_{\delta\downarrow0}\limsup_{j\to\infty}
 \diam_{d_j}(\overline B_\delta^E)=0,
 \qquad
 \lim_{\delta\downarrow0}\diam_d(\overline B_\delta^E)=0,
\]
then \(\sup_{M\times M}|d_j-d|\to0\). In particular the identity
correspondence gives Gromov--Hausdorff convergence of these compact spaces.
\end{lemma}
\begin{proof}
Fix \(0<\delta<A\). On \(M\), collapse \(K=\overline B_\delta^E\)
to one point and define the quotient length pseudodistances \(q_j,q\)
by charging the usual length on the annulus and zero for travel through
\(K\). A competitor needs to visit the collapsed point at most once:
discard everything between its first arrival and last departure. Thus a
near-minimizer consists either of an annular curve, or of two annular
curves and one free jump through \(K\). Replacing that jump by a curve
whose length is arbitrarily close to the distance between its endpoints
shows
\[
 0\le d_j-q_j\le\diam_{d_j}(K),\qquad
 0\le d-q\le\diam_d(K).
\]
The replacement curve is allowed to leave \(K\), but remains in \(M\),
as required by the definition of the intrinsic distance. Conversely,
discarding the portion of any ordinary curve between its first and last
visits to \(K\) proves \(q_j\le d_j\) and \(q\le d\).
On the fixed annulus, smooth convergence gives a length comparison
\((1-\epsilon_j)\operatorname{Length}_h\le
\operatorname{Length}_{h_j}\le
(1+\epsilon_j)\operatorname{Length}_h\), with \(\epsilon_j\to0\).
It follows that \((1-\epsilon_j)q\le q_j\le(1+\epsilon_j)q\).
Consequently
\[
 \sup|d_j-d|\le
 \diam_{d_j}(K)+\diam_d(K)+\epsilon_j\diam_d(M).
\]
First take the upper limit as \(j\to\infty\), and then let
\(\delta\downarrow0\).
\end{proof}

\subsection{Identification of all metric blow-downs}

\begin{theorem}[All metric blow-downs]\label{lim:tangents}
Let \(R_j\to\infty\), set
\[
 s_j=\log R_j^2,\quad \sigma_j=L+s_j,\quad
 \tau_j=\log(\sigma_j/L),\quad \ell_j=\rho(R_j),
 \qquad D_{R_j}(z)=R_jz.
\]
If \(\alpha(\tau_j)\to\alpha_\infty\) smoothly on
\(Y=\CP^{n-1}\), then
\begin{equation}\label{lim:metric-blowdown}
 (\C^n,\ell_j^{-2}g,0)
 \longrightarrow(C_{\alpha_\infty},o)
 \quad\text{in pointed Gromov--Hausdorff topology}.
\end{equation}
Under \(D_{R_j}\), the rescaled metrics converge smoothly on every
compact subset of \(\C^n\setminus\{0\}\). Every asymptotic cone
arises in this way, and every \(\alpha\in\mathcal A\) is realized.
\end{theorem}
\begin{proof}
Let \(g_j=\ell_j^{-2}D_{R_j}^*g\), a complete metric on a fixed copy
of \(\C^n\). On a fixed closed coordinate annulus the exterior part
of its potential is exactly
\begin{align}\label{lim:potential-rescaling}
 \ell_j^{-2}U_{\mathrm{ext}}(R_jz)
 ={}&|z|^{2a}
 \left(1+\frac{\log|z|^2}{\sigma_j}\right)^b\notag\\
 &\cdot\exp\left[
 \phi\left(\tau_j+
 \log\left(1+\frac{\log|z|^2}{\sigma_j}\right),[z]\right)
 \right].
\end{align}
The added phase and all its spatial derivatives tend uniformly to zero.
Uniform mixed derivative bounds for the path, together with the smooth
dependence of \(\phi_\alpha\) on \(\alpha\), imply that
\eqref{lim:potential-rescaling} converges smoothly to
\[
 U_{\alpha_\infty}(z)=|z|^{2a}e^{\phi_{\alpha_\infty}([z])}.
\]
The cap correction \(h(s_j+\log|z|^2)/\ell_j^2\), and all its
fixed-order derivatives, tend to zero by Proposition~\ref{cap:global}.
The added constant has no effect on the metric. This proves smooth
convergence of \(g_j\) to the cone metric \(g_\infty\) on annuli.

Fix \(A>0\) and use intrinsic distances on \(\overline B_A^E\).
The radial paths in Proposition~\ref{lim:complete} give, for
\(0<\delta<\min\{1,A\}\),
\begin{equation}\label{lim:small-core}
 \limsup_{j\to\infty}
 \diam_{g_j}^{\mathrm{intr}}(\overline B_\delta^E)
 \le C\delta^a.
\end{equation}
Indeed, the fixed core contributes \(D/\ell_j\to0\), and, whenever
\(\delta R_j\ge1\),
\[
 \frac{\rho(\delta R_j)}{\rho(R_j)}
 =\delta^a
 \left(1+\frac{\log\delta^2}{\sigma_j}\right)^{b/2}
 \le\delta^a.
\]
The cone has the analogous estimate by homogeneity. Its puncture is
completed by a single point, whereas each \(g_j\) is already smooth
there. Lemma~\ref{lim:core-lemma} now proves uniform convergence of
intrinsic distance functions on \(\overline B_A^E\).

It remains to relate these intrinsic distances to the ambient ones.
For each fixed \(A>1\), \eqref{lim:distance} gives
\[
 \inf_{|z|=A}d_{g_j}(0,z)
 \ge c\frac{\rho(AR_j)}{\rho(R_j)}-\frac D{\ell_j}.
\]
Since $\rho(AR_j)/\rho(R_j)\to A^a$, it follows that
\begin{equation}\label{lim:confinement}
 \liminf_{j\to\infty}\inf_{|z|=A}d_{g_j}(0,z)\ge cA^a.
\end{equation}
The cone has the same type of lower bound. Given a metric radius
\(Q>0\), choose \(A\) so large that both boundary distances are
greater than \(10Q\), for all sufficiently large \(j\). Every
minimizing geodesic between two points of \(\overline B_{g_j}(0,Q)\)
stays in \(B_{g_j}(0,3Q)\): its length is at most \(2Q\).
It consequently stays in \(B_A^E\). The same argument applies to the
cone. Thus ambient and intrinsic distances agree for the relevant points.
The uniform intrinsic convergence identifies the radius functions as
well. Balls of radius \(Q\) are then approximated using the balls of
radii \(Q\pm\epsilon_j\); geodesic segments from the origin show
that their Hausdorff discrepancy is at most \(O(\epsilon_j)\).
This proves \eqref{lim:metric-blowdown} on every bounded metric ball.

Finally \(\rho\) is strictly increasing to infinity by
\eqref{lim:rho-derivative}. Hence every sequence of distance scales
\(\ell_j\to\infty\) can be written uniquely as \(\ell_j=\rho(R_j)\)
for sufficiently large \(j\). The path has relatively compact image
and cluster set exactly \(\mathcal A\), by Lemma~\ref{geom:path}.
After passing to a subsequence its phases converge to some
\(\alpha_\infty\in\mathcal A\), and the preceding argument applies.
Conversely, a phase subsequence approaching any prescribed member of
\(\mathcal A\) gives a sequence \(R_j\to\infty\). This proves both
exhaustion and realization.
\end{proof}

\subsection{Ricci currents, volume growth, and curvature decay}

\begin{proposition}[Ricci currents under blow-down]\label{lim:ricci}
Set \(f_g=-\log\det(\mathcal U_{j\bar k})\). Along the subsequence
in Theorem~\ref{lim:tangents}, the normalized potentials
\begin{equation}\label{lim:normalized-ricci}
 F_j(z)=f_g(R_jz)-\kappa s_j+nb\log\sigma_j
\end{equation}
converge in \(L^1_{\mathrm{loc}}(\C^n)\) and smoothly off zero to
\[
 f_{\alpha_\infty}(z)=
 \kappa\log|z|^2+\psi_{\alpha_\infty}([z])-(n+1)\log a.
\]
Consequently
\begin{equation}\label{lim:ricci-current}
 D_{R_j}^*\Ric(\omega)\rightharpoonup
 \ddc f_{\alpha_\infty}
\end{equation}
as currents on \(\C^n\). Off the vertex this is precisely
\(\pi^*\alpha_\infty=\Ric(\omega_{\alpha_\infty})\).
\end{proposition}
\begin{proof}
The determinant expansion and cap estimates give, for \(|w|\ge1\),
\begin{equation}\label{lim:global-f-asymptotic}
 f_g(w)=\kappa\log|w|^2-nb\log(L+\log|w|^2)+O(1).
\end{equation}
The remainder is uniformly bounded; on a fixed intermediate annulus this
also follows from smooth positivity. At infinity its precise expression
is \(\Psi(\tau,[w])+E+J\), where \(\Psi=\psi-(n+1)\log a\)
and \(E,J\to0\) with the needed derivatives. Thus smooth convergence
to the stated potential on punctured compact sets follows directly,
including the normalization constant.

For convergence across the origin we give a uniform integrable majorant.
Write \(t=\log|z|^2\) and take \(s_j\ge1\). If \(|R_jz|\ge1\),
then \(s_j+t\ge0\), and \eqref{lim:global-f-asymptotic} gives
\[
 F_j(z)=\kappa t-nb\log\left(1+\frac t{\sigma_j}\right)+O(1).
\]
If \(t\ge0\), the absolute value of the logarithm is at most \(t/L\).
If \(-s_j\le t<0\), it is at most \(|t|/L\), since
\[
 \log\frac{\sigma_j}{\sigma_j+t}
 =\int_{\sigma_j+t}^{\sigma_j}\frac{du}{u}
 \le\frac{|t|}{L}.
\]
If instead \(|R_jz|<1\), the smooth function \(f_g\) is bounded on
\(\overline B_1^E\), and \(-t>s_j\). Therefore
\[
 |F_j(z)|\le C+\kappa s_j+nb\log(L+s_j)
 \le C'(1+|t|).
\]
Combining the cases proves, on every fixed coordinate ball and almost
everywhere,
\begin{equation}\label{lim:log-majorant}
 |F_j(z)|\le C(1+|\log|z||).
\end{equation}
The right-hand side is locally integrable for the Euclidean measure.
Dominated convergence gives \(L^1_{\mathrm{loc}}\) convergence, and
continuity of distributional differentiation proves
\eqref{lim:ricci-current}. In particular no current supported only at
the vertex can be lost in this argument.

There is no factor \(\ell_j^{-2}\) in \eqref{lim:ricci-current}:
\(\Ric(c\omega)=\Ric(\omega)\) for every constant \(c>0\).
In fact, the complex Hessian of the rescaled potential is
$\ell_j^{-2}R_j^2(\mathcal U_{p\bar q}(R_jz))$, so its
negative logarithmic determinant is exactly
\[
 f_g(R_jz)+2n\log\ell_j-2n\log R_j=F_j(z).
\]
\end{proof}

\begin{proposition}[Exact asymptotic volume ratio]\label{lim:volume}
Let \(v_{2n}\) be the Euclidean unit-ball volume. Then
\begin{equation}\label{lim:avr}
 \lim_{R\to\infty}\frac{\Vol_g B_g(0,R)}{v_{2n}R^{2n}}=a^n.
\end{equation}
Moreover the global, all-center noncollapsing estimate
\begin{equation}\label{lim:noncollapse}
 \Vol_g B_g(p,R)\ge a^n v_{2n}R^{2n}
 \qquad(p\in\C^n,\ R>0)
\end{equation}
holds.
\end{proposition}
\begin{proof}
Choose a subsequence as in Theorem~\ref{lim:tangents}. The volume estimate
in \eqref{lim:core-bounds}, with \(g_j=\ell_j^{-2}D_{R_j}^*g\), gives
\begin{equation}\label{lim:core-volume}
 \limsup_{j\to\infty}\Vol_{g_j}(B_\delta^E)
 \le C\delta^{2na}\qquad(0<\delta<1).
\end{equation}
Indeed, for \(\delta R_j\ge1\) the left side before taking limits is
bounded by
\[
 \frac D{\ell_j^{2n}}+
 C\delta^{2na}
 \left(1+\frac{\log\delta^2}{\sigma_j}\right)^{nb}
 \le\frac D{\ell_j^{2n}}+C\delta^{2na}.
\]
When \(\delta R_j<1\), the fixed core volume divided by
\(\ell_j^{2n}\) suffices.

By the confinement and distance convergence just proved, all unit metric
balls lie in one fixed coordinate ball \(B_A^E\), and their radius
functions converge uniformly there wherever needed. On
\(\overline B_A^E\setminus B_\delta^E\), the volume forms converge
smoothly. The limiting unit sphere is the smooth hypersurface
\(U_{\alpha_\infty}=1\), so it has zero limiting volume. Thus
convergence of indicator functions almost everywhere and convergence of
the smooth densities imply convergence of unit-ball volumes outside
\(B_\delta^E\). Letting \(\delta\downarrow0\) and using
\eqref{lim:core-volume} proves
\[
 \ell_j^{-2n}\Vol_g B_g(0,\ell_j)
 \longrightarrow
 \Vol_{C_{\alpha_\infty}} B(o,1)=a^n v_{2n},
\]
where the last equality is the Hopf-cone volume formula
\eqref{geom:cone-volume}. This argument uses ordinary smooth
convergence and the explicit core estimate, rather than a separate
Ricci-limit volume convergence theorem.

Since \(g\) is complete with \(\Ric(g)\ge0\), Bishop--Gromov
comparison states that
\(\Vol_g B_g(p,R)/(v_{2n}R^{2n})\) is nonincreasing in \(R\);
see, for example, \cite{Petersen2016}. Its limit at \(p=0\) therefore
exists, and the computed subsequence identifies it as \(a^n\).
For arbitrary \(p\), set \(d=d_g(0,p)\). The inclusions
\[
 B_g(0,R-d)\subset B_g(p,R)\subset B_g(0,R+d)
 \qquad(R>d)
\]
show that the limiting ratio at \(p\) is also \(a^n\). Monotonicity
then implies \eqref{lim:noncollapse} at every finite radius.
\end{proof}

\begin{corollary}[Quadratic curvature decay]\label{lim:curvature}
The constructed metric satisfies
\[
 |\operatorname{Rm}(g)|_g(z)
 \le\frac C{(1+d_g(0,z))^2}.
\]
\end{corollary}
\begin{proof}
For every sufficiently large \(R\), the metrics
\(\rho(R)^{-2}D_R^*g\) on \(\{1/2\le|z|\le2\}\) are uniformly
positive, with uniformly bounded derivatives of every fixed order.
This follows directly from \eqref{lim:potential-rescaling}, the bounded
mixed derivatives of the path, and the cap estimates; no phase subsequence
is needed for these bounds. Their curvature norms are uniformly bounded.
Scaling back at \(|z|=R\) gives
\(|\operatorname{Rm}(g)|_g\le C\rho(R)^{-2}\).
Use \eqref{lim:distance} to replace \(\rho(R)\) by distance for large
\(R\); smoothness on the remaining compact set completes the proof.
\end{proof}

\subsection{Proofs of the main results}

\begin{proof}[Proof of Theorem~\ref{main:theorem}]
Lemma~\ref{geom:path} supplies the path with prescribed cluster set.
Proposition~\ref{cap:global} gives a globally smooth K\"ahler metric
with nonnegative Ricci curvature, and Proposition~\ref{lim:complete}
proves its completeness. Theorem~\ref{lim:tangents} realizes and exhausts
all the asserted asymptotic cones. Proposition~\ref{lim:ricci} identifies
their Ricci currents, and Proposition~\ref{lim:volume} gives the exact
volume ratio and uniform noncollapsing. The pointed isometry criterion
in Theorem~\ref{geom:classification} gives the stated quotient of \(\mathcal A\).
In particular, whenever \(\mathcal A\) has two inequivalent members,
the complete manifold has nonisometric tangent cones at infinity.
\end{proof}

\begin{proof}[Proof of Theorem~\ref{main:nonunique}]
Use the compact connected family from Corollary~\ref{geom:nonunique}
as the input \(\mathcal A\) in Theorem~\ref{main:theorem}.
Its endpoint cones are not pointed isometric, and both occur as
tangent cones at infinity. The theorem supplies all the other
asserted properties.
\end{proof}

\section{Acknowledgements}
The authors would like to thank their supervisors, Wenshuai Jiang and Gang Liu, for suggesting this interesting and challenging problem. The authors also want to thank professor G\'abor Sz\'ekelyhidi for drawing our attention to the related work of Arka Karmakar \cite{Karmakar2026}. The authors first learned of the problem two years ago and have since discussed it frequently with them. The problem was fully resolved only recently, and the authors are deeply grateful for their constant encouragement and support.

This work was supported by the National Key Research and Development Program of
China (No. 2022YFA1005500).

\end{document}